\documentclass[final,3p,times,authoryear]{elsarticle}

\journal{Statistics and Probability Letters}

\usepackage{amsmath,amssymb,amsthm,bbm}
\usepackage[usenames,dvipsnames]{color}

\usepackage[shortlabels]{enumitem}
\numberwithin{equation}{section}
\usepackage{verbatim}

\newcommand{\dto}{%
	\mathrel{\vbox{\offinterlineskip\ialign{%
				\hfil##\hfil\cr
				$\scriptstyle d$\cr
				$\longrightarrow$\cr
			}}}}

\newcommand{\vto}{%
	\mathrel{\vbox{\offinterlineskip\ialign{%
				\hfil##\hfil\cr
				$\scriptstyle v$\cr
				$\longrightarrow$\cr
			}}}}

\newcommand{\toi}{\to\infty}

\newcommand{\pr}{\mathbb{P}} 
\newcommand{\ex}{\mathbb{E}}

\newcommand{\Rset}{\mathbb{R}} 
\newcommand{\Nset}{\mathbb{N}} 
\newcommand{\Zset}{\mathbb{Z}}

\newcommand{\N}{\mathbb{N}}

\newcommand{\spaceX}{\mathbb{X}} 

\newcommand{\x}{\spaceX}
\newcommand{\borel}{\mathcal{B}}
\newcommand{\bborel}{\mathcal{B}_b}
\newcommand{\bounded}{\bborel}
\newcommand{\subbounded}{\mathcal{C}_b}
\newcommand{\mx}{\mathcal{M}(\spaceX)}

\newcommand{\mpx}{\mathcal{M}_p(\mathbb{X})}

\newcommand{\cbb}{CB_b(\spaceX)}
\newcommand{\cbbp}{CB_b^+(\spaceX)}

\newcommand{\sub}{\subseteq}

\newcommand{\bx}{\boldsymbol{x}}

\newcommand{\bzero}{\boldsymbol{0}}

\newcommand{\loo}{l_{0,0}}

\newcommand{\cl}[1]{\overline{#1}}

\newcommand{\1}[1]{\mathbbm{1}_{#1}}

\newtheorem{theorem}{Theorem}[section]

\newtheorem{proposition}[theorem]{Proposition}

\theoremstyle{remark}
\newtheorem{remark}[theorem]{Remark}
\newtheorem{example}[theorem]{Example}

\begin{document}

\begin{frontmatter}



\title{A note on vague convergence of measures
}


\author[a]{Bojan Basrak}
\ead{bbasrak@math.hr}
\address[a]{Department of Mathematics, Faculty of Science, University of Zagreb, Bijeni\v cka 30, Zagreb, Croatia}

\author[a]{Hrvoje Planini\'c\corref{cor1}}
\ead{planinic@math.hr}
\cortext[cor1]{Corresponding author}

\begin{abstract}
	We propose a new approach to vague convergence of measures 
based on the general theory of boundedness due to \cite{hu:1966}.
The article explains how this connects and unifies
several frequently used types of vague convergence from the literature.
Such an approach allows one to translate already developed results from one type of vague convergence to another.
We further analyze the corresponding notion of vague topology and 
give a new and useful characterization of 
 convergence in distribution of random measures in this topology.

\end{abstract}


\end{frontmatter}


\section{Introduction}
Let $\x$ be a Polish space, i.e.\ separable topological space which is metrizable by a complete metric. Denote by $\borel(\x)$ the corresponding Borel $\sigma$--field and choose a subfamily $\bounded(\x)\subseteq \borel(\x)$ of sets, called \textit{bounded} (Borel) sets of $\x$. When there is no fear of confusion, we will simply write $\borel$ and $\bounded$.



A Borel measure $\mu$ on $\x$ is said to be \textit{locally (or boundedly) finite} if $\mu(B)<\infty$ for all $B\in \bborel$. Denote by $\mx=\mathcal{M}(\x,\bborel)$ the space of all such measures.  For measures $\mu,\mu_1,\mu_2,\ldots \in \mx$, we say that $\mu_n$ converge \textit{vaguely} to $\mu$  and denote this by $\mu_n\vto \mu$, if as $n\toi$,
\begin{equation}\label{eq:vague_convergence}
\mu_n(f)=\int f d\mu_n \to \int f d\mu= \mu(f) \, ,
\end{equation}  
for all bounded and continuous real--valued functions $f$ on $\x$ with support being a bounded set. Denote by $\cbb$ the family of all such functions and by $\cbbp$ the subset of all nonnegative functions in $\cbb$.

In \cite[Section 4]{kallenberg:2017},  theory of vague convergence is developed under the assumption that $\bounded$ is the family all metrically bounded Borel sets w.r.t.\ a metric generating the topology of $\x$. The same notion of convergence was studied before under the name of $w^{\#}$-convergence in~\cite[Section
A2.6]{daley:verejones:2008t1}.
However, some other related notions appeared in the literature, especially in the context of extreme value theory, see~\cite{lindskog:resnick:roy:2014} 
and references therein.



In Section \ref{sec:vague} we propose a new, and in our opinion, more intuitive approach to the concept of vague convergence which 
also clarifies the connection between {all those other notions} of convergence found in the literature. It  is based on the abstract theory of boundedness due to~\cite{hu:1966} which allows one to characterize all metrizable families of bounded sets. In this way, the emphasis is put on the choice of the family of bounded sets rather than on the construction of the appropriate metric. 


In Section \ref{sec:vague_topology}, we briefly discuss the question of metrizability of the so--called vague topology on $\mx$. This topology is usually defined using projection maps such that convergence of sequences of measures derived in this way, exactly corresponds to vague convergence defined in (\ref{eq:vague_convergence}). It turns out that this topology is metrizable (and moreover Polish), as it has been observed in the literature, but we were not able to find a complete argument, so we provide one in Proposition \ref{prop:metrization}.

Finally, in Section \ref{sec:conv_in_distr}, we  
deduce sufficient conditions for convergence in distribution of random measures on $\spaceX$ with respect to the vague topology. In particular, for random measures $N, N_1, N_2, \dots$ on $\spaceX$, we show that convergence of so--called Laplace functionals 
$\ex[e^{-N_n(f)}]\to \ex[e^{-N(f)}] $ for all $f$ in a certain class of Lipschitz continuous functions (see Proposition \ref{cor:Lipschitz} for details), implies that random measures $N_n$ converge in distribution to $N$. Recently, \cite{lohr:rippl2016} provided conditions under which a subalgebra of $\cbb$ is convergence determining 
for vague convergence of (deterministic) measures in $\mx$. An application of Proposition \ref{cor:Lipschitz} allows us to extend their results to the random case, see Proposition \ref{prop:Lohr_Rippl}.

\section{The abstract concept of bounded sets}
\label{sec:vague}
Following \cite[Section V.5]{hu:1966}, we say that a family of sets $\bounded\subseteq \borel(\x)$ is a (Borel) \textit{boundedness} in $\x$ if (i)  $A\sub B \in \bounded$ for $A\in\borel$ implies $A\in \bounded$; (ii) $A,B\in \bounded$ implies $A\cup B\in \bounded$.   
A subfamily $\subbounded$ of $\bounded$ is called a \textit{basis} of $\bounded$ if every $B\in \bounded$ is contained in some $C\in \subbounded$.
Finally, boundedness $\bounded$ is said to be \textit{proper} if it is adapted to the topology of $\x$ in the sense that  for each $B \in \bounded$ there exists an open set $U\in \bounded$  such that $\cl{B} \subseteq U$, where  $\cl{B}$ denotes the closure of $B$ in $\x$. 

To be consistent with the existing terminology of \cite[p.\ 19]{kallenberg:2017}, we say that a boundedness \textit{properly localizes} $\x$ if it is proper and has a countable basis which covers $\x$.
If $d$ metrizes $\x$ then the family of all sets in $\borel$ with finite $d$--diameter is an example of a boundedness which properly localizes $\x$. By \cite[Corollary 5.12]{hu:1966}, this turns out to be the only example.
\begin{theorem}[{\cite[Corollary 5.12]{hu:1966}}]\label{thm:Hu}
Boundedness $\bounded$ properly localizes $\x$ if and only if 
there exists a metric on $\x$ which generates the topology of $\x$ and under which the metrically bounded Borel subsets of $\x$ coincide with $\bounded$.
\end{theorem}

\begin{remark}
Note that {\cite[Corollary 5.12]{hu:1966}} concerns boundednesses which also contain non--Borel subsets of $\x$. We restrict to Borel subsets  since we work with Borel measures and it is easily seen that in this setting the conclusion of  {\cite[Corollary 5.12]{hu:1966}} still holds.
\end{remark}

Observe, if $\bounded$ properly localizes $\x$ then one can find a basis $(K_m)_{m\in N}$ of $\bounded$ which consists of open sets  and satisfies
\begin{equation}\label{eq:proper_boundedness}
	\cl{K}_m\subseteq K_{m+1} \,, \quad \mbox{	for all $m\in \Nset$\,.}
	\end{equation}
	Indeed, fix a metric which generates $\bounded$ and take $(K_m)$ to be open balls around some fixed point and with radius strictly increasing to infinity; see also \cite[Lemma 5.9]{hu:1966} for a direct argument which does not rely on Theorem \ref{thm:Hu}.
	 Conversely, if $(K_m)$ is a sequence of open sets covering $\x$ and such that   (\ref{eq:proper_boundedness}) holds, then the boundedness $\bounded=\{B\in \borel : \exists m\in \N \text{ such that } B\subseteq K_m\}$  properly localizes $\x$. We refer to $(K_m)$ as a \textit{proper localizing sequence}.	 
	 
Observe, by simply choosing a different family of bounded sets in $\x$ one alters the space of locally finite measures $\mx$ as well as the corresponding notion of vague convergence defined in (\ref{eq:vague_convergence}). We mention a couple of important examples.
%
	\begin{example}[{Weak convergence}]\label{exa:weak}
By taking $\bounded$ to be the family of all Borel subsets of $\spaceX$, we end up with the usual notion of weak convergence of finite measures on $\spaceX$.
Such a family $\bounded$ properly localizes $\spaceX$ since taking $K_m=\spaceX$ for all $m\in \Nset$ yields a proper localizing sequence. 

Observe that it is easy to find a metric which metrizes $\x$ and generates $\bounded$ in this case. Simply choose any bounded metric which generates the topology of $\spaceX$.
\end{example}

\begin{example}[{Vague convergence of Radon measures}]
When the space $\spaceX$ is additionaly locally compact, by choosing $\bounded$ as the family of all relatively compact Borel subsets of $\spaceX$ we obtain the well known notion of vague convergence of Radon measures on $\spaceX$ as described in \cite{kallenberg:1983} or~\cite{resnick:1987}. Recall, a set is 
relatively compact if its closure is compact. Note that in this case, since $\spaceX$ is locally compact, second countable and Hausdorff, one can find a sequence $(K_m)_{m\in \Nset}$ of relatively compact open subsets of $\spaceX$ which cover $\spaceX$ and satisfy \eqref{eq:proper_boundedness}. In particular, these $K_m$'s form a basis for $\bounded$ and hence $\bounded$ properly localizes $\spaceX$.
\end{example}

\begin{example}[{Hult-Lindskog convergence}]\label{exa:M0Convergence} Let $(\spaceX',d')$ be a complete and separable metric space. In the theory of regularly varying random variables and processes, the sets of interest, i.e.\ bounded sets, are usually those which are actually bounded away from some fixed closed set  $\mathbb{C}\subseteq \mathbb{X}'$.
 
 More precisely, 
 assume that $\spaceX$ is of the form $\spaceX=\spaceX'\setminus \mathbb{C}$ equipped with the subspace topology and
 set $\bounded$ to be the class of all Borel sets $B\subseteq \spaceX$ such that for some $\epsilon>0$, $d'(x,\mathbb{C})>\epsilon$ for all $x\in B$, where   $d'(x,\mathbb{C})=\inf\{d'(x,z):z\in \mathbb{C} \}$. In this way, we obtain the notion of the so--called $\mathbb{M}_{\mathbb{O}}$-convergence (where $\mathbb{O}=\spaceX$) as discussed in~\cite{lindskog:resnick:roy:2014} and originally introduced by~\cite{hult:lindskog:2006}. Observe, such $\bounded$ properly localizes $\x$ since one can take $K_m=\{x\in \spaceX: d'(x,\mathbb{C}) >1/m\}$, $m\in \N$, as a proper localizing sequence.
\end{example}

\begin{remark}
As observed by \cite[p.\ 125]{kallenberg:2017}, under the notation of the previous example, one metric $d$ which is topologically equivalent to $d'$ and generates $\bounded$  is given by
\[
d(x,y)=\left(d'(x,y)\wedge 1\right) \vee
\left|1/d'(x,\mathbb{C}) - 1/d'(y,\mathbb{C})\right| \; , \; x,y\in \mathbb{X} \, .
\]
In fact, this construction illustrates the basic idea in the proof of 
Theorem \ref{thm:Hu}, see
 \cite[Theorem 5.11]{hu:1966}.
\end{remark}

\begin{remark}
By the proof of \cite[Theorem 5.11]{hu:1966}, if $\x$ is completely metrizable and $\bounded$ properly localizes $\x$, one can assume that the metric which generates $\bounded$ in Theorem~\ref{thm:Hu} is also complete.
\end{remark}

In the rest of the paper we will always assume that the space $\spaceX$ is properly localized by a  family of bounded sets $\bounded$.
In such a case, Theorem \ref{thm:Hu} 
ensures that one can directly translate results in \cite{kallenberg:2017} to the vague convergence of locally finite measures on $\x$. In particular, by the so--called Portmanteau theorem (see \cite[Lemma 4.1]{kallenberg:2017}), $\mu_n\vto \mu$ in $\mx$ is equivalent to convergence
\begin{align}\label{eq:vagueConv_sets}
\mu_n(B)\to \mu(B)
\end{align} 
for all $B\in \bborel$ with $\mu(\partial B)=0$, where $\partial B$ denotes the boundary of the set $B$.  

\subsection{Vague convergence of point measures}	
Denote by
$\delta_{x}$ the Dirac measure concentrated at $x\in\mathbb{X}$.	 A \textit{(locally finite) point measure} on $\x$ is a measure $\mu\in \mx$ which is of the form $\mu=\sum_{i=1}^K \delta_{x_i}$ for some $K\in \{0,1,\dots\}\cup\{\infty\}$ and (not necessarily distinct) points $x_1,x_2,\dots,x_K$ in $\spaceX$. Note that by definition at most finitely many  $x_i$'s fall into every bounded set $B\in \bborel$. Denote by $\mpx$ the space of all point measures on $\x$.  Equivalently, one can define   point measures as  integer--valued measures in $\mx$, see e.g.~\cite[Exercise 3.4.2]{resnick:1987}.

The following result, which is a simple consequence of (\ref{eq:vagueConv_sets}), characterizes vague convergence in the case of point measures. 
It is fundamental when applying continuous mapping arguments to results on convergence in distribution of point processes (i.e.\ random point measures), see e.g.\ the proof of \cite[Theorem 7.1]{resnick:2007}. 
The proof of necessity can be found in \cite[Proposition 3.13]{resnick:1987}, and sufficiency is proved similarly, we omit the details; cf. also \cite[Lemma 2.1]{dombry:eyi-minko:2016}.

\begin{proposition}\label{prop:pointMeasures}
Let $\mu,\mu_1,\mu_2,\ldots  \in \mpx$ be  point measures. Then  $\mu_n \vto \mu$ implies that for every $B\in \bborel$ such that $\mu(\partial B)=0$ there exist $k,n_0\in \Nset$ and points $x_i^{(n)},x_i$, $n\geq n_0, i=1,\dots,k$, in $B$ such that for all $n\geq n_0$,
\begin{align*}
\mu_n\vert_{B} = \sum_{i=1}^k \delta_{x_i^{(n)}} \; \; \text{and} \; \; \mu\vert_{B} = \sum_{i=1}^k \delta_{x_i} \: ,
\end{align*}
and for all $i=1,\dots,k$, 
\begin{align*}
x_i^{(n)}\to x_i \; \text{ in } \; \x \, ,
\end{align*}
where $\mu_n\vert_{B}$ and $\mu\vert_{B}$ denote restriction of measures $\mu_n$ and $\mu$, respectively, to the set $B$. 
Conversely, to show that $\mu_n\vto \mu$, it is sufficient to check convergence of points in sets $B$ from  any basis $\mathcal{C}_b$ of $\bborel$.
\end{proposition}

\section{A comment on metrizability of the vague topology}\label{sec:vague_topology}
 
To study random measures in $\mx$ and $\mpx$, it is
 	useful to go beyond the concept of convergent sequences and 
 	specify a topology on those spaces as well. Moreover, to exploit the full power of probabilistic methods it is desirable that such a topology remains Polish.
The standard choice of topology on $\mx$ (see e.g.\ \cite[p.\ 111]{kallenberg:2017} or \cite[p.\ 140]{resnick:1987}) is the smallest topology under which the  maps $\mu \mapsto \mu(f)$ are continuous for all $f\in \cbbp$. Equivalently, this is the topology obtained by taking sets of the form 
\begin{equation}\label{eq:basis_vague}
\left\{\nu \in \mx : |\mu(f_i)- \nu (f_i)|< \epsilon \; \text{for all} \; i=1,\dots, k\right\}
\end{equation}
for $\epsilon>0, k\in \Nset$ and  $f_1,\dots,f_k \in \cbbp$, to be the neighborhood base of $\mu \in \mx$. We call this topology the \textit{vague} topology. 
Note that, by definition, $\mu_n \longrightarrow \mu$ with respect to the vague topology if and only if $\mu_n \vto \mu$. 

It is shown in  \cite[Theorem 4.2]{kallenberg:2017} that there exists a metric $\rho$ on $\mx$ with the property that $\mu_n\vto\mu$ in $\mx$ if and only if $\rho(\mu_n,\mu)\to 0$. Moreover, the metric space $(\mx,\rho)$ is complete and  separable.

It is now tempting to  conclude at once that the topology generated by the metric $\rho$ coincides with the vague topology and consequently that the vague topology is Polish.  However, there exist different topologies with the same convergent sequences, hence one can not identify those two topologies without knowing a priori that the vague topology, i.e.\ the topology generated by the sets in (\ref{eq:basis_vague}), is \textit{sequential}, i.e.\ completely determined by its converging sequences, see \cite{franklin:1965} (cf.~also~\cite{dudley:1964}). Note that  any first countable and hence any metrizable space is sequential. Since we have not been able to find  
 an appropriate argument anywhere  in the literature, we provide one here omitting some technical details.

\begin{proposition}\label{prop:metrization}
The space $\mx$ equipped with the vague topology is metrizable and hence Polish.
\end{proposition}
\begin{proof}[Sketch of the proof]
Consider the space $\mathcal{\hat{M}}(\spaceX)$ of all finite Borel measures on $\spaceX$, i.e.\ Borel measures $\mu$ on $\spaceX$ such that $\mu(\spaceX)<\infty$. Equip $\mathcal{\hat{M}}(\spaceX)$ with the smallest topology under which the maps $\mu\mapsto \mu(f)$ are continuous for all nonnegative and bounded continuous functions $f$ on $\spaceX$. This topology is usually called  the \textit{weak} topology. Also, let $\mathcal{\hat{M}}_1(\spaceX)\subseteq \mathcal{\hat{M}}(\spaceX)$ be the subset of all probability measures equipped with the relative topology. 

 By \cite[Appendix III, Theorem 5]{billingsley:1968}, there exists a metric $\hat{\rho}_1$ on $\mathcal{\hat{M}}_1(\spaceX)$ (the so--called Prohorov metric) which generates the weak topology and moreover  $\hat{\rho}_1$ is bounded by $1$. Using this result, it is possible, but cumbersome, to show that the function $\rho$ on $\mathcal{\hat{M}}(\spaceX)\times \mathcal{\hat{M}}(\spaceX)$ given by
$$\hat{\rho}(\mu,\nu)=|\mu(\spaceX) - \nu(\spaceX)| + \left(\mu(\spaceX) \wedge \nu(\spaceX)\right) \cdot \hat{\rho}_1(\mu(\,\cdot\,)/\mu(\spaceX),\nu(\,\cdot\,)/\nu(\spaceX) )$$
is a proper metric which generates the weak topology on $\mathcal{\hat{M}}(\spaceX)$. 

Further, take a proper localizing sequence $(K_m)_{m\in \N}$. In particular (\ref{eq:proper_boundedness}) holds, and since $\x$ is a metric space, for every $m\in \N$ one can find a continuous function $g_m$ on $\x$ such that $\1{\overline{K}_m}\leq g_m \leq \1{K_{m+1}}$. Clearly, $(g_m)\subseteq \cbbp$. 
For every $m\in \N$ define a mapping $T_m: \mx \to\mathcal{\hat{M}}(\spaceX) $  such that 
$T_m(\mu)$ is the (unique) measure satisfying $T_m(\mu)(f)= \mu(f\cdot g_m)$ for all nonnegative and bounded functions $f$ on $\x$. Note that, if $\mu \neq \nu$ for $\mu,\nu \in \mx$ then necessarily $T_m(\mu)\neq T_m(\nu)$ for some $m\in \N$. Using the fact that $\hat{\rho}$ generates the weak topology on $\mathcal{\hat{M}}(\spaceX)$, one can show that the metric
\begin{align*}
\tilde \rho (\mu,\nu) = \sum_{m=1}^\infty  \frac{ 1 \wedge \hat \rho (T_m(\mu),T_m(\nu))}{2^m} \, , \; \mu,\nu \in \mx
\end{align*}
generates the vague topology on $\mx$.

\end{proof} 

	\section{Sufficient conditions for convergence in distribution of random measures}\label{sec:conv_in_distr}
A \textit{random measure} on $\x$ is a random element in $\mx$ w.r.t.\ the smallest  $\sigma$--algebra under which the maps 
$\mu \mapsto \mu(B)$ are measurable for all $B\in\bborel$.  By \cite[Lemma 4.7]{kallenberg:2017}, this $\sigma$--algebra equals the Borel $\sigma$--algebra on $\mx$ arising from the vague topology. 
Convergence in distribution in $\mx$ 
is considered w.r.t.\ the vague topology and is denoted by "$\dto$".

For random measures $N,N_1,N_2,\dots$ on $\spaceX$, it is fundamental that $N_n\dto N$ in $\mx$ if and only if $N_n(f)\dto N(f)$ in $\Rset$ for all $f\in \cbbp$. This is further equivalent to convergence of Laplace functionals $\ex[e^{-N_n(f)}]\to \ex[e^{-N(f)}]$ for all $f\in \cbbp$, see  \cite[Theorem 4.11]{kallenberg:2017}.
We first show that in the last two convergences it is sufficient to consider only functions which are Lipschitz continuous with respect to any suitable metric.

\subsection{Lipschitz functions determine convergence in distribution}

For any metric $d$ on $\spaceX$ denote by $LB^{+}_b (\spaceX,d)$ the family of all bounded nonnegative functions $f$ on $\spaceX$ which have bounded support and are Lipschitz continuous with respect to $d$.
Furthermore, for a set $B\subseteq \spaceX$ and $\epsilon>0$ denote 
\[
B^{\epsilon} = B^{\epsilon}(d)=\{x\in \spaceX:\: d(x,B)\leq \epsilon\} \; .
\]
\begin{proposition}\label{cor:Lipschitz}
Assume that $d$ is a metric on $\spaceX$ which generates the corresponding topology and such that for any $B\in \mathcal{B}_b$ there exists an $\epsilon>0$ such that $B^{\epsilon}\in\mathcal{B}_b$. Then $N_n\dto N$ in $\mx$ if and only if $
\ex[e^{-N_n(f)}]\to \ex[e^{-N(f)}]$ for every $f\in LB_b^+(\spaceX,d)
$.
\end{proposition}
\begin{remark}
Every metric on $\spaceX$ which generates the topology and the family of bounded sets (as in Theorem \ref{thm:Hu}) satisfies the assumptions of the previous proposition.
\end{remark}

\begin{example}\label{exa:M0}
Consider the case from Example \ref{exa:M0Convergence}. Recall, $(\spaceX',d')$ is assumed to be a complete and separable metric space and $\mathbb{C}\subseteq \spaceX'$ a closed subset of $\spaceX'$. The space $\spaceX=\spaceX' \setminus \mathbb{C}$ is equipped with the subspace topology (i.e.\ generated by $d'$) and with $B\subseteq \spaceX$ being bounded if and only  if $B$ is contained in $\{x\in\spaceX : d'(x,\mathbb{C})>1/m\}$ for some $m\in \Nset$.
In this case, the metric $d'$ generates the topology but the corresponding class of metrically bounded sets does not coincide with $\bounded$. Still, $d'$ obviously satisfies the assumptions of the previous proposition. 
\end{example}

\begin{proof}[Proof of Proposition \ref{cor:Lipschitz}] We only need to prove sufficiency.
Take arbitrary $k \in \Nset$, $\lambda_1,\dots,\lambda_k\geq 0$ and 
$B_1,\dots,B_k \in \bborel$ such that $\pr(N(\partial B)=0)=1$ for all $i=1,\dots,k$. By \cite[Theorem 4.11]{kallenberg:2017},
  the result will follow if we show that 
\begin{equation}\label{eq:Laplace_convergence_simple_fn}
\lim_{n\to \infty} \ex[e^{-\sum_{i=1}^k\lambda_i N_n(B_i)}]= \ex[e^{-\sum_{i=1}^k\lambda_i N(B_i)}] \; .
\end{equation} 
For all $m\in\Nset, i=1,\dotsc k$ and $x\in \spaceX$ set
\begin{align}\label{eq:approx_functions}
f_{m,i}^+(x)=1- (md(x,\overline{B}_i)\wedge 1)\; ,\; f_{m,i}^-(x)=md(x,(B_i^\circ)^c)\wedge 1 \, ,
\end{align}
$B^\circ$ denotes the interior of $B$.
Using the elementary fact that for a closed set $C\subseteq \spaceX$ and $x\in \spaceX$, $x\in C$ if and only if $d(x,C)=0$, it is straightforward to show that for all $i=1,\dots, k$ as $m\to \infty,$
\begin{align}\label{eq:approx_with_Lip_fn}
f_{m,i}^+ \searrow \1{\overline{B}_i} \;  \text{and} \; f_{m,i}^- \nearrow \1{B_i^\circ} \, .
\end{align}

Functions $f_{m,i}^-$ obviously have bounded support for all $i$ and $m$, and since $f_{m,i}^+ \leq \1{\{x\in \spaceX \: : \: d(x,\overline{B}_i)\leq 1/m\}}$, by assumption on the metric $d$, $f_{m,i}^+$ has bounded support for $m$ large enough. Assume without loss of generality that this is true for all $m\in \Nset$. 
Further, it is not difficult to show that for all $i, m$ and all $x,y\in \spaceX$
\[
|f_{m,i}^+(x)-f_{m,i}^+(y)| \vee |f_{m,i}^-(x)-f_{m,i}^-(y)| \leq m d(x,y) \;.
\]
Hence, $f_{m,i}^+$ and $f_{m,i}^-$ are elements of $LB_b^+(\spaceX,d)$.

By the monotone and the dominated convergence theorem,
\begin{equation*}
\lim_{m\to\infty}\ex[e^{-\sum_{i=1}^k\lambda_i N( f_{m,i}^-)}]= \ex[e^{-\sum_{i=1}^k\lambda_i N(B^{\circ}_i)}]=\ex[e^{-\sum_{i=1}^k\lambda_i N(B_i)}]\; , 
\end{equation*}
where the last equality follows since we assumed $N(\partial B_i)=0$ a.s. for all $i$. Since $\sum_{i=1}^k\lambda_i f_{m,i}^-$ is again in $LB_b^+(\spaceX,d)$, convergence in (\ref{eq:Laplace_convergence_simple_fn}) will follow if we prove that
\begin{equation}\label{eq:approximationNn(A)}
\lim_{m\to\infty}\limsup_{n\to\infty} \left|\ex[e^{-\sum_{i=1}^k\lambda_i N_n(B_i)}]-\ex[e^{-\sum_{i=1}^k\lambda_i N_n(f_{m,i}^-)}]\right|=0 \;. 
\end{equation}
Since for all $i$ and $m$, $f_{m,i}^-\leq \1{B_i}\leq f_{m,i}^+$,
\begin{align*}
0\leq \ex[e^{-\sum_{i=1}^k\lambda_i N_n(f_{m,i}^-)}]-\ex[e^{-\sum_{i=1}^k\lambda_i N_n(B_i)}]
\leq \ex[e^{-\sum_{i=1}^k\lambda_i N_n(f_{m,i}^-)}]-\ex[e^{-\sum_{i=1}^k\lambda_i N_n(f_{m,i}^+)}] \, .
\end{align*}
Notice that for for fixed $i$ and all $m\in \Nset $,  the support of $f_{m,i}^+$  is contained in the support of $f_{1,i}^+$, which we assumed is a bounded set. Since $N$ a.s.\ puts finite measure on such sets, applying the dominated convergence theorem twice yields that $\lim_{m\to\infty}\ex[e^{-\sum_{i=1}^k\lambda_i N( f_{m,i}^+)}]= \ex[e^{-\sum_{i=1}^k\lambda_i N(\overline{B}_i)}]$.
Therefore,
\begin{multline*}
\lim_{m\to\infty}\limsup_{n\to\infty} \left|\ex[e^{-\sum_{i=1}^k\lambda_i N_n(B_i)}]-\ex[e^{-\sum_{i=1}^k\lambda_i N_n(f_{m,i}^-)}]\right|
\\
\leq  \lim_{m\to\infty}\ex[e^{-\sum_{i=1}^k\lambda_i N(f_{m,i}^-)}]-\ex[e^{-\sum_{i=1}^k\lambda_i N(f_{m,i}^+)}]\\
=\ex[e^{-\sum_{i=1}^k\lambda_i N(B_i^\circ)}]-\ex[e^{-\sum_{i=1}^k\lambda_i N(\overline{B}_i)}]=0 \, ,
\end{multline*}
where the last equality holds since $N(\partial B_i)=0$ a.s. Hence, (\ref{eq:approximationNn(A)}) holds and this finishes the proof.
\end{proof}

\begin{remark}
One can easily verify that the previous result holds with the family of Lipschitz functions replaced by any family $\mathcal{C}\subseteq \cbbp$ such that (i) $\alpha f + \beta g \in \mathcal{C}$ for all $f,g\in\mathcal{C},\alpha,\beta\geq 0$, (ii) for all $B$ in a dissecting semi-ring $\mathcal{I}\subseteq \{B\in \bborel : \pr(N(\partial B)=0)=1 \}$ , there exist $(f_m^+)_{m\in\Nset},(f_m^-)_{m\in\Nset} \subseteq \mathcal{C}$ such that $f_m^+ \searrow \1{\overline{B}}$ and $f_m^- \nearrow \1{B^{\circ}}$; see \cite{kallenberg:2017} for the definition of a dissecting semi-ring. In \cite{kallenberg:2017}, a family $\mathcal{C}\subseteq \cbbp$ satisfying condition (ii) is called an approximating class of $\mathcal{I}$ and it was shown that such families determine vague convergence in $\mx$, see \cite[Lemma 4.1]{kallenberg:2017}. 
\end{remark}

\subsection{An application to multiplicatively closed subsets of $\cbbp$}

Recently, \cite{lohr:rippl2016} obtained a very general result which gives sufficient conditions for a family of functions to be convergence determining for vague convergence of measures in $\mx$, see \cite[Theorem 2.3]{lohr:rippl2016}. Their approach rests on the Stone-Weierstrass theorem. Using Proposition \ref{cor:Lipschitz} we are able to extend the results from \cite{lohr:rippl2016} to the case of random measures.

\begin{proposition}\label{prop:Lohr_Rippl}
Let $\mathcal{C}\subseteq \cbbp$  be such that
\begin{enumerate}[(i)]
\item $\alpha f + \beta g \in \mathcal{C}$ and $f\cdot g\in \mathcal{C}$ for all $f,g\in\mathcal{C},\alpha,\beta\geq 0$; 
\item the family $\mathcal{C}$ induces the topology of $\spaceX$, i.e.\ sets $f^{-1}(U)$ for $U\subseteq \Rset_+$ open and $f\in \mathcal{C}$, form a subbase for the original topology on $\spaceX$;
\item \label{item:hm's} for every $m \in \Nset$ there exists $\delta_m>0$ and $h_m\in \mathcal{C}$ such that $h_m\geq \delta_m \1{K_m}$. 
\end{enumerate}
Then
$N_n\dto N$ in $\mx$ if and only if $\ex[e^{-N_n(f)}]\to \ex[e^{-N(f)}]$ for every $f\in \mathcal{C}$.
\end{proposition}

\begin{proof}
Again, we only need to prove sufficiency. Since $\spaceX$ is separable one can find a countable subfamily $(f_i)_{i\in \Nset} \subseteq \mathcal{C}$ of functions with values in $[0,1]$  which induce the topology of $\spaceX$, see \cite[Theorem 2.3, Step 2]{lohr:rippl2016}.  Assume that the functions $h_m$ from \ref{item:hm's} are bounded by $1$ and that $(h_m)_m\subseteq (f_i)_i$. Define the metric $d$ on $\spaceX$ by $d(x,y)=\sum_{i\in \Nset}2^{-i} |f_i(x)-f_i(y)|$ for all $x,y\in \spaceX$. In particular, the topology generated by the metric $d$ coincides with the original topology on $\spaceX$. We now show that metric $d$ satisfies the remaining assumption of Proposition \ref{cor:Lipschitz}.

Take an arbitrary $B\in \bborel$ and find $m\in \Nset$ such that $B\subseteq K_m$. Since the function $h_m$ has bounded support, there exist an $m'\geq m$ such that $y\notin K_{m'}$ implies that $h_m(y)=0$. Hence, for $i_{m}\in \Nset$ such that $f_{i_m}=h_m$ we have that $y\notin K_{m'}$  implies that $d(x,y)\geq 2^{-i_m}h_m(x) \geq 2^{-i_m} \delta_m=: 2\epsilon$ for all $x\in B$. Hence, $B^\epsilon \subseteq K_{m'}$ and therefore $B^\epsilon\in \bborel$.   

Denote now by $\text{span}(\mathcal{C})$ the vector subspace of $\cbb$ spanned by $\mathcal{C}$. Since for fixed $f,g\in\mathcal{C}$, $\ex[e^{-\alpha N_n(f) - \beta N_n(g)}]\to \ex[e^{-\alpha N(f) - \beta N(g)}]$ for all $\alpha,\beta\geq 0$, by the the continuity theorem for Laplace transforms (see \cite[15.5.2]{kallenberg:1983}), random variables $N_n(f)$ and $N_n(g)$ converge jointly in distribution to $N(f)$ and $N(g)$. In particular, $N_n(f)\dto N(f)$ for all
$f\in \text{span}(\mathcal{C})$. 

Fix an arbitrary $g\in LB^{+}_b (\spaceX,d)$ and let $m$ be such that $x\notin K_m$ implies $g(x)=0$. In particular, by assumption \ref{item:hm's}, $h_m(x)=0$ implies $g(x)=0$. Now by an application of the Stone-Weierstrass theorem (see Lemma 2.6 and the proof of Theorem 2.3 in \cite{lohr:rippl2016}) it follows that for all $\epsilon>0$ there exists an $f\in \text{span}(\mathcal{C})$ such that
\begin{align}\label{eq:Mointer-1}
|g(x)-f(x)|\leq \epsilon h_{m}(x) \,  \: \text{for all} \; x\in \spaceX \, .
\end{align}
Using this we will show that $\big|\ex[e^{-N_n(g)}] - \ex[e^{-N(g)}]\big| \to 0$ 
 and since $g\in LB^{+}_b (\spaceX,d)$ was arbitrary, Proposition \ref{cor:Lipschitz} will imply that $N_n\dto N$ in $\mx$.


Take $\epsilon>0$ and let $f\in \text{span}(\mathcal{C})$ be such that (\ref{eq:Mointer-1}) holds. Since $\ex[e^{-N_n(f)}]\to \ex[e^{-N(f)}]$,
\begin{align}\label{eq:Mointer0}
\limsup_{n\toi} \big|\ex[e^{-N_n(g)}] - \ex[e^{-N(g)}] \big| \leq  \limsup_{n\toi} \big|\ex[e^{-N_n(g)}] - \ex[e^{-N_n(f)}] \big|+ \big|\ex[e^{-N(f)}] - \ex[e^{-N(g)}] \big|  \, .
\end{align}
Using (\ref{eq:Mointer-1}) and the simple bound $|e^{-x} - e^{-y}|\leq |x-y|\wedge 1$, $x,y\geq 0$, we obtain that for any random measure $M$ in $\mx$
\begin{align*}
\big|\ex[e^{-M(g)}] - \ex[e^{-M(f)}] \big| \leq \ex[M(|g-f|)\wedge 1] \leq \pr(M(h_m)> C) + \epsilon C \, , 
\end{align*}
for all $C>0$. Using this bound for $N_n$ and $N$, (\ref{eq:Mointer0}) yields 
that for all $C,\epsilon>0$
\begin{align}\label{eq:inter1}
\limsup_{n\toi} \big|\ex[e^{-N_n(g)}] - \ex[e^{-N(g)}] \big| \leq \limsup_{n\toi}  \pr(N_n(h_m)> C) + \pr(N(h_m)> C) + 2\epsilon C \, .
\end{align}

Since $h_m \in \mathcal{C}$, it holds that $N_n(h_m)\dto N(h_m)$ which implies that the random variables $(N_n(h_m))_n$ are tight.  Thus, letting $\epsilon\to 0$ and then $C\toi$ in (\ref{eq:inter1}) yields that $\big|\ex[e^{-N_n(g)}] - \ex[e^{-N(g)}]\big| \to 0$ as $n\toi$.


%
%
\end{proof}

\begin{remark}
\cite[Theorem 2.3]{lohr:rippl2016} allow the  family $\mathcal{C}$ to include functions with unbounded support and deduce sufficient conditions for vague convergence of measures in $\mx$ which integrate functions from $\mathcal{C}$. Based on personal communication with Wolfgang L\"ohr, we note that such an extension of Proposition \ref{prop:Lohr_Rippl} is possible but omit the details.

\end{remark}

\section*{Acknowledgements}
The authors would like to thank Adam Jakubowski for fruitful discussions during his visit to Zagreb and Philippe Soulier for careful reading of the manuscript and several helpful comments. Finally, the authors are grateful to Wolfgang L\"ohr for a very useful discussion which led to addition of Proposition \ref{prop:Lohr_Rippl} to our paper.
 The research of both authors is supported in part by the HRZZ project "Stochastic methods in analytical and applied problems" (3526) and the SNSF/HRZZ Grant CSRP 2018-01-180549. 

\section*{}
  \bibliographystyle{elsarticle-harv} 
  \bibliography{bib_note_on_vague}


%
%
%

\end{document}